\documentclass[letterpaper, 10 pt, conference]{ieeeconf}  

\usepackage{amsmath,amscd}

\usepackage{bm}
\usepackage{amsthm,array}
\usepackage{amsmath}
\usepackage{amsfonts,latexsym}
\usepackage{graphicx,subfig,wrapfig}
\usepackage{times}
\usepackage{psfrag,epsfig}
\usepackage{verbatim}
\usepackage{tabularx}
\usepackage{color}
\usepackage{soul}
\usepackage{indentfirst}
\usepackage{inputenc}
\usepackage[capitalise]{cleveref}
\usepackage{courier}
\usepackage{authblk}
\usepackage{mathtools}
\usepackage{mathtools}
\usepackage{cite}
\usepackage{amssymb}
\usepackage{graphicx}
\usepackage{caption}
\usepackage{dblfloatfix}

\newcommand{\nth}[1]{#1^{\mathrm{th}}}

\def\0{\boldsymbol{0}}
\def\a{\boldsymbol{a}}
\def\b{\boldsymbol{b}}

\def\x{\boldsymbol{x}}
\def\y{\boldsymbol{y}}

\def\v{\boldsymbol{v}}
\def\p{\boldsymbol{p}}
\def\q{\boldsymbol{q}}

\def\I{\text{I}}

\def\R{\mathbb{R}}

\def\Ad{\text{Ad}}

\long\def\answer#1{}

\long\def\comment#1{}

\theoremstyle{definition}

\newtheorem{prop}{Proposition}

\newtheorem{assumption}{Assumption}

\newcolumntype{P}[1]{>{\centering\arraybackslash}p{#1}}
\newcolumntype{M}[1]{>{\centering\arraybackslash}m{#1}}

\IEEEoverridecommandlockouts                              

\title{\LARGE \bf
Decentralized and Recursive Identification for Cooperative Manipulation of Unknown Rigid Body with Local Measurements}

\author{Taosha Fan \hspace{3em} Huan Weng \hspace{3em} Todd Murphey
\thanks{Taosha Fan, Huan Weng and Todd Murphey are with the Department of Mechanical Engineering, Northwestern University,
         Evanston, IL 60201, USA
        {\tt\small \{taosha.fan, huan\_weng\}@u.northwestern.edu, t-murphey@northwestern.edu}.}%
}

\begin{document}

\maketitle
\thispagestyle{empty}
\pagestyle{empty}

\begin{abstract}
This paper proposes a fully decentralized and recursive approach to online identification of unknown kinematic and dynamic parameters for cooperative manipulation of a rigid body based on commonly used local measurements. To the best of our knowledge, this is the first paper addressing the identification problem for 3D rigid body cooperative manipulation, though the approach proposed here applies to the 2D case as well. In this work, we derive truly linear observation models for kinematic and dynamic unknowns whose state-dependent uncertainties can be exactly evaluated. Dynamic consensus in different coordinates and a filter for dual quaternion are developed with which the identification problem can be solved in a distributed way. It can be seen that in our approach all unknowns to be identified are time-invariant constants. Finally, we provide numerical simulation results to illustrate the efficacy of our approach indicating that it can be used for online identification and adaptive control of rigid body cooperative manipulation.\par
\textit{Index-Terms}--- Cooperative manipulation; 3D rigid body; distributed identification; dynamic consensus; dual quaternion.
\end{abstract}

\section{Introduction}
Multi-robot cooperative manipulation has seen great progress in the last several decades. In the 2D planar scenario, multi-robotic system demonstrates its ability to transport times larger and heavier loads under various conditions \cite{pereira2004decentralized,fink2008multi,murphey2008adaptive}. Recently further attention has been paid to the more general 3D cooperative manipulation of a rigid body, such as multi-finger grasping \cite{li1989grasping}, motion planing to transport a large object \cite{yamashita2000motion}, impedance control for cooperating manipulators based on internal force \cite{bonitz1996internal}, multiple quadrotors with a suspended rigid-body load \cite{wu2014geometric} and robust cooperative manipulation without force and torque information \cite{verginis2016robust}.\par 

{In most of these works, even though some assert that they are adaptive and can deal with uncertainties, it is usually assumed that either kinematic parameters (relative position and orientation) or dynamic parameters (mass, mass center, inertia tensor) are at least partially known, or robots implicitly communicate with a central processing unit to help decision making. These assumptions may be problematic in practice --- it is impossible to always have prior knowledge of the unknown load while implicit communication with a central unit is only available under limited circumstances and essentially reduces the overall distributeness of the system.} As a result, a suitable identification approach to estimating unknown kinematic and dynamic parameters is needed, which benefits multi-robot cooperative manipulation.\par

Though decentralized parameter identification for planar cooperative manipulation has been studied in \cite{franchi2015decentralized} using velocities and forces measured in inertia frame, it is difficult to implement these methods for systems involving a rigid body due to the complexity of dynamics and inconvenience of processing inertia frame measurements. Moreover, forces and torques are practically measured or estimated in a local reference frame and thus estimation of the relative orientation and position from the local frame to inertia frame -- which is typically time-varying -- is additionally required to transform local forces and torques to the inertia frame.\par

In this paper, we  propose a fully decentralized and recursive approach to identifying the kinematic and dynamic unknowns for cooperative manipulation of a 3D rigid body. To the best of our knowledge, similar problems have not been addressed before. The approach proposed can be used for planar cooperative manipulation as well. An advantage of our approach is that the identification only relies on local measurements,\footnote{We refer to measurements in a local reference frame as ``\textit{local measurements}'' and those in an inertial frame as ``\textit{global measurements}''.} the benefits of which are three-fold: i) it is consistent with robotic manipulation where control laws and forces are applied in local reference frame; ii) the rigid body dynamics in a local reference frame are more concise and thus the identification is simplified; iii) it can be shown that all the kinematic and dynamic unknowns to be estimated are constant and no estimation on time-varying parameters/states is needed. 
The other contributions of this paper include the derivation of linear observation models with evaluable state-dependent uncertainties, dynamic consensus in different coordinates and appropriate filtering of dual quaternions for our specific problem.\par

The rest of this paper is organized as follows: \cref{section::note} defines the most frequently used notations in the paper. \cref{section::preliminary} briefly reviews quaternions and dual quaternions that are used to develop linear observation model for pose estimation. \cref{section::problem} formulates the identification problem with common assumptions in cooperative manipulation and in \cref{section::measure} linear observation models for kinematic and dynamic unknowns are derived. \cref{section::filter} discusses state-dependent uncertainties evaluation, dynamic consensus in different coordinates and filtering on dual quaternions so that the identification problem can be properly solved in a distributed way. Numerical results are given in \cref{section::num} and conclusions are made in \cref{section::conclusion}.
\section{Nomenclature}\label{section::note}
\vspace{0.5em}
\noindent
\begin{tabular}{M{2.5cm}@{}p{.2cm}@{}p{5.3cm}}
$\mathcal{W}$& & Inertia frame. \\
$\mathcal{S}_i$& & Sensor frame associated with robot $i$. \\
\hline
$g_{ji}\in SE(3)$& & Rigid body transformation matrix from $\mathcal{S}_i$ to $\mathcal{S}_j$. \\
$R_{ji}\in SO(3)$& & Rotational part for $g_{ji}$ . \\
$\bm{t}_{ji}\in \R^3$ & &Translational part for $g_{ji}$ . \\
\hline
$\widehat{\bm{x}}_{ji}\in \widehat{Q}$& & Unit dual quaternion for $g_{ji}$. \\
$\widetilde{\bm{q}}^r_{ji}\in Q$& & Rotational part for $\widehat{\bm{x}}_{ji}$. \\
$\widetilde{\bm{q}}^d_{ji}\in Q$& & Translational part for $\widehat{\bm{x}}_{ji}$. \\
\hline
$\bm{f}_i\in\R^3$ & &Force applied by robot $i$ at $\mathcal{S}_i$. \\
$\bm{\tau}_i\in\R^3$& & Torque applied by robot $i$ robot at $\mathcal{S}_i$. \\
$\bm{F}_i\in\R^3$ & & Equivalent total force applied by all robots at $\mathcal{S}_i$. \\
$\bm{T}_i\in\R^3$ & & Equivalent total torque applied by all robots at $\mathcal{S}_i$. \\
\hline
$\bm{\omega}_i\in\R^3$& & Body-fixed angular velocity of $\mathcal{S}_i$. \\
$\bm{v}_i\in\R^3$& & Body-fixed linear velocity of $\mathcal{S}_i$. \\
$\bm{\xi}_i\in\R^6$& & $\bm{\xi}_i=\begin{bmatrix}
\bm{\omega}_i^T & \bm{v}_i^T
\end{bmatrix}^T$. \\
\hline
$\bm{\alpha}_i\in\R^3$& & Body-fixed angular acceleration of $\mathcal{S}_i$. \\
$\bm{a}_i\in\R^3$ & &Body-fixed linear acceleration of $\mathcal{S}_i$. \\
$\overline{\bm{a}}_i\in\R^3$& & Body-fixed linear proper acceleration of $\mathcal{S}_i$. \\
\hline
$\bm{p}_c^i\in\R^3$& & Mass center w.r.t. $\mathcal{S}_i$. \\
$m\in \R^+$ & &Mass of the rigid body\\
$\mathcal{I}_i\in\R^{3\times 3}$& & Inertia tensor evaluated at $\p_c^i$ w.r.t. $\mathcal{S}_i$\\
$\mathcal{I}_{xx}^i$, $\mathcal{I}_{yy}^i$, $\mathcal{I}_{zz}^i$, $\mathcal{I}_{xy}^i$, $\mathcal{I}_{xz}^i$, $\mathcal{I}_{yz}^i$ & & Components of $\mathcal{I}_i$\\
\hline
$\bm{g}\in\R^3$ & &Gravity acceleration in $\mathcal{W}$. \\
\end{tabular}
\section{Preliminaries}\label{section::preliminary}
In this section, we give a brief review of \textit{quaternions} and \textit{dual quaternions} that are often used to represent $SO(3)$ and $SE(3)$. A more detailed introduction to quaternions and dual quaternions can be found in \cite{daniilidis1996dual,srivatsan2016estimating}. In this paper, $Q$ and $\widehat{Q}$ are respectively used to denote quaternion and dual quaternion. 
\subsection{Quaternion}
A quaternion $\widetilde{\bm{q}}=(q_0,\,\bm{q})\in Q$ is a 4-tuple where $q_0\in \R$ is the scalar part and $\bm{q}\in \R^3$ the vector part. 
The multiplication $\odot$ of two quaternions $\widetilde{\bm{p}}$ and $\widetilde{\bm{q}}$ is defined as 
\begin{equation}\label{eq::quat_mult}
\begin{aligned}
\widetilde{\bm{p}}\odot \widetilde{\bm{q}} &= (p_0 q_0 - \bm{p} \cdot \bm{q}, p_0 \bm{q} + q_0 \bm{p} + \bm{p} \times \bm{q}).
\end{aligned}
\end{equation}
Furthermore, linear operators $(\cdot)^+:Q\rightarrow \R^{4\times 4}$ and $(\cdot)^-:Q\rightarrow \R^{4\times 4}$ associated with \cref{eq::quat_mult} are defined as
\begin{equation}\label{eq::quatLR}
\widetilde{\bm{q}}^+ =\begin{bmatrix}
q_0 & -\bm{q}^T\\
\bm{q} & {\bm{q}}^\times + q_0 \I
\end{bmatrix}, \,
\widetilde{\bm{q}}^- =\begin{bmatrix}
q_0 & -\bm{q}^T\\
\bm{q} & -{\bm{q}}^\times + q_0 \I
\end{bmatrix}
\end{equation}
where $(\cdot)^\times: \R^3\rightarrow \R^{3\times 3}$ is a linear operator such that $\bm{a}^\times\bm{b}=\bm{a}\times\bm{b} $ and $\I\in\R^{3\times 3}$ is the identity matrix, then 
$$\widetilde{\bm{p}}\odot \widetilde{\bm{q}} = \widetilde{\bm{p}}^+ \cdot \widetilde{\bm{q}} = \widetilde{\bm{q}}^- \cdot \widetilde{\bm{p}}.$$ 
The conjugate $\widetilde{\bm{q}}^*$ of a quaternion $\widetilde{\bm{q}}$ is  
$$\widetilde{\bm{q}}^* =(q_0,\,-\bm{q}) $$ and $\widetilde{\q}\widetilde{\q}^*=\widetilde{\q}^*\widetilde{\q}=(\|\widetilde{\q}\|^2,\,\0)$.\par 

Unit quaternions $\widetilde{\bm{q}}$ are quaternions with $\|\widetilde{\bm{q}}\|=1$ and can be used to represent $SO(3)$ such that
\begin{equation}\label{eq::quaternion}
\nonumber
\widetilde{\bm{q}} = (\cos\dfrac{\theta}{2},\sin\dfrac{\theta}{2}\bm{\omega})
\end{equation}
where $\theta\in [-\pi,\,\pi]$ is the angle and the unit vector $\bm{\omega}\in\R^3$ is the rotational axis. Besides for unit quaternions we have $$\widetilde{\bm{q}}\odot\widetilde{\bm{q}}^* = \widetilde{\bm{q}}^*\odot\widetilde{\bm{q}} =(1,\,\bm{0}).$$ 
Let $\b'\in\R^3$ be obtained by rotating $\b\in\R^3$ with a unit quaternion $\widetilde{\bm{q}}$ and then we have
$$\widetilde{\b}' = \widetilde{\bm{q}}\odot \widetilde{\b}\odot \widetilde{\bm{q}}^*$$
where $\widetilde{\b}=(0,\,\bm{b})$ and by the conjugate property of unit quaternion it is equivalent to
\begin{equation}\label{eq::quat_lin}
\widetilde{\bm{q}}\odot \widetilde{\b} = \widetilde{\b}'\odot  \widetilde{\bm{q}}.
\end{equation}
\subsection{Dual Quaternion}
A dual quaternion $\widehat{\bm{x}}=\widetilde{\bm{p}}+\epsilon\widetilde{\bm{q}}\in\widehat{Q}$ where $\widetilde{\bm{p}}, \widetilde{\bm{q}}\in Q$ are quaternions and $\epsilon$ is defined to be $\epsilon\neq 0$ and $\epsilon^2=0$. The multiplication $\otimes$ of two dual quaternions $\widehat{\bm{x}}_1 = \widetilde{\bm{p}}_1 +\epsilon\widetilde{\bm{q}}_1$ and $\widehat{\bm{x}}_2 = \widetilde{\bm{p}}_2+\epsilon\widetilde{\bm{q}}_2$ is given by
$$\widehat{\bm{x}}_1\otimes \widehat{\bm{x}}_2 = \widetilde{\bm{p}}_1\odot \widetilde{\bm{p}}_2+\epsilon(\widetilde{\bm{p}}_1\odot \widetilde{\bm{q}}_2+\widetilde{\bm{q}}_1\odot \widetilde{\bm{p}}_2).$$ 
Similarly, linear operators $(\cdot)^+$ and $(\cdot)^-:\widehat{Q} \rightarrow \R^{8\times 8}$ for dual quaternions are defined by
\begin{equation}\label{eq::duatLR}
\widehat{\bm{x}}^+=\begin{bmatrix}
\widetilde{\bm{p}}^+ & \mathbf{O}\\
\widetilde{\bm{q}}^+ & \widetilde{\bm{p}}^+
\end{bmatrix}, \hspace{1.5em}
\widehat{\bm{x}}^-=\begin{bmatrix}
\widetilde{\bm{p}}^- & \mathbf{O}\\
\widetilde{\bm{q}}^- & \widetilde{\bm{p}}^-
\end{bmatrix}
\end{equation} 
so that
\begin{equation}\label{eq::duatLR2}
\widehat{\bm{x}}_1\otimes\widehat{\bm{x}}_2 = \widehat{\bm{x}}_1^+\cdot \widehat{\bm{x}}_2 = \widehat{\bm{x}}_2^-\cdot \widehat{\bm{x}}_1.
\end{equation}
 Dual quaternions have three conjugates which are respectively $\widehat{\bm{x}}^{1*}=\widetilde{\bm{p}}-\epsilon\widetilde{\bm{q}}$, $\widehat{\bm{x}}^{2*}=\widetilde{\bm{p}}^*+\epsilon\widetilde{\bm{q}}^*$ and $\widehat{\bm{x}}^{3*}=\widetilde{\bm{p}}^*-\epsilon\widetilde{\bm{q}}^*$.\par
A dual quaternion $\widehat{\bm{x}}$ is unit if $\widehat{\bm{x}}\otimes \widehat{\bm{x}}^{2*}=1$ which can be used to represent $SE(3)$. Given $g=(R,\,\bm{t})\in SE(3)$ where $R\in SO(3)$ is rotation and $\bm{t}\in \R^3$ the translation, then we may use unit dual quaternion $\widehat{\bm{x}}=\widetilde{\q}_r+\epsilon\widetilde{\q}_d$ to represent $g$ where $\widetilde{\bm{q}}_r$ is a unit quaternion corresponding to the rotation $R\in SO(3)$ and $\widetilde{\q}_d$ is a quaternion corresponding to the translation $\bm{t}\in \R^3$ that is given by
$$\widetilde{\q}_d = \dfrac{\widetilde{\bm{t}}\odot\widetilde{\bm{q}}_r}{2}.$$ 
The rigid body transformation of a point $\b\in \R^3$ given by a unit dual quaternion $\widehat{\bm{x}}$ is
\begin{equation}\label{eq::transform}
\widehat{\b}'=\widehat{\bm{x}}\otimes \widehat{\b}\otimes \widehat{\bm{x}}^{3*}
\end{equation}
where $\widehat{\b}=1+\epsilon \widetilde{\b}$ and $\widetilde{\b}=(0,\,\b)$. It is known for unit dual quaternions that $\widehat{\bm{x}}^{1*}\otimes\widehat{\bm{x}}^{3*}=\widehat{\bm{x}}^{3*}\otimes\widehat{\bm{x}}^{1*}=1$ so that \cref{eq::transform} can be rewritten as
\begin{equation}\label{eq::dquat_lin}
\widehat{\bm{x}}\otimes \widehat{\b} = \widehat{\b}'\otimes \widehat{\bm{x}}^{1*}.
\end{equation}
\cref{eq::quat_lin,eq::dquat_lin} are often used to derive linear observation models to estimate orientation and pose \cite{srivatsan2016estimating,choukroun2006novel}.

\begin{figure}
\includegraphics[trim =0mm 0mm 0mm 0mm,width=0.4\textwidth]{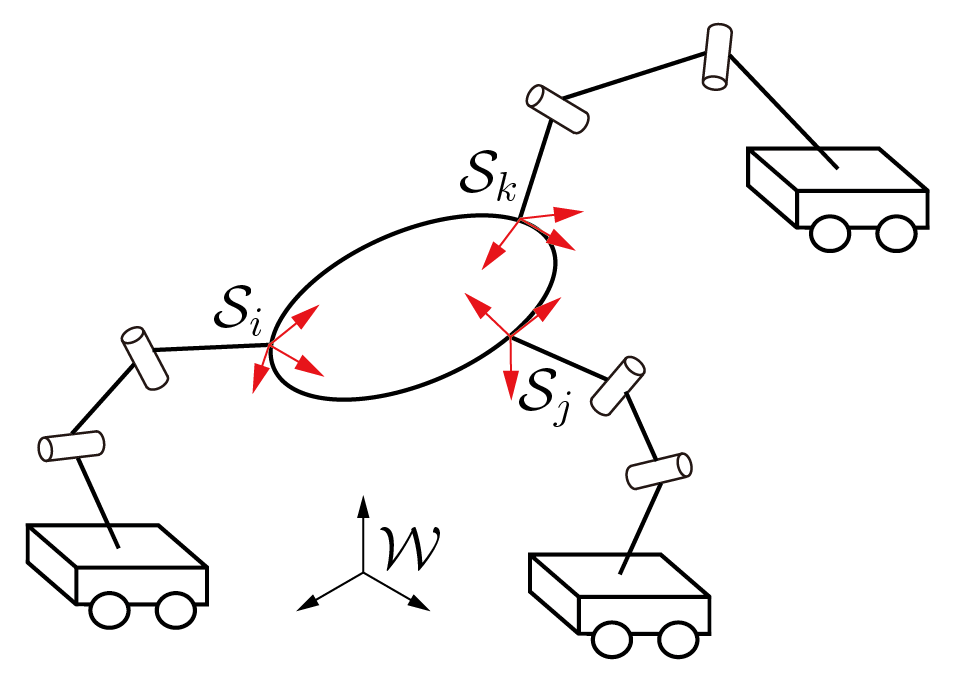}
\caption{Cooperative Manipulation of a 3D rigid body.}\label{fig::demo}
\end{figure}
\section{Problem Formulation}\label{section::problem}
\subsection{Problem Statement}
We consider the problem that a network of $n$ robots manipulate a rigid body load as shown in \cref{fig::demo}. The network is defined as an undirected graph $G=(V,\,E)$ where $V=\{1,\,2,\,3,\,\cdots,\, n\}$ is the node set of robots and $E\subset V\times V$ the edge set of communication links. In this paper we make the following assumptions for our identification problem.
\begin{assumption}\label{assump::1}
The network $G$ is connected and each robot $i$ can only communicate with its one-hop neighbours $N_i=\{j\in V|\,(i,\,j)\in E\}$.
\end{assumption}
\begin{assumption}\label{assump::2}
The end-effector of each robot is fixed with the rigid body.
\end{assumption}
\begin{assumption}\label{assump::3}
Each sensor frame $\mathcal{S}_i$ is fixed with the end-effector of robot $i$ as well as the rigid body whose origin is the contact point of robot $i$ with the rigid body. 
\end{assumption}
\noindent\textit{Remark:} The rigid body transformation $g_{ji}$ between different sensor frame $\mathcal{S}_i$ and $\mathcal{S}_j$ is not priorly known and needs to be identified so that each robot does not know the contact point, position and orientation for the other robots.
\begin{assumption}\label{assump::4}
The body-fixed linear velocity $\v_i$, body-fixed angular velocity $\bm{\omega}_i$, the body-fixed proper acceleration $\overline{\a}_i$ at $\nth{i}$ contact point and force $\bm{f_i}$ and torque $\bm{\tau}_i$ applied by robot $i$ to the rigid body are measurable. All of $\v_i$, $\bm{\omega}_i$, $\overline{\a}_i$, $\bm{f_i}$ and $\bm{\tau}_i$ are measured in sensor frame $\mathcal{S}_i$.
\end{assumption}
\noindent \textit{Remark:} The \textit{proper acceleration} $\overline{\a}_i$ is the acceleration relative to a free-fall which is measurable by a accelerometer. The relationship between proper acceleration $\overline{\a}_i$ and the usual \textit{coordinate acceleration} $\a_i$ is 
\begin{equation}\label{eq::prop_acc}
\overline{\a}_i = \a_i - R_i^T\bm{g}
\end{equation}
where $R_i\in SO(3)$ is the rotation matrix from $\mathcal{S}_i$ to $\mathcal{W}$.\par
\vspace{0.5em}
The overall identification problem addressed in this paper is formulated as follows.\\
\vspace{-0.5em}
\newline
\noindent\textbf{Problem}\hspace{1em} Suppose a group of robots manipulate a rigid body load with \cref{assump::1}-4 hold, we define a distributed and recursive algorithm so that each robot $i$ can identify 
\begin{enumerate}
\item The rigid body transformation $g_{ji}$ where $j\in N_i$
\item The mass $m$, mass center $\p_c^i$ and inertia tensor $\mathcal{I}_i$. 
\end{enumerate} 
\par
\vspace{0.5em}
\noindent\textit{Remark:} The rigid body transformation $g_{ji}$, mass $m$, mass center $\p_c^i$ and inertia tensor $\mathcal{I}_i$ are all \textit{time-invariant constants} in the identification problem.\par
\vspace{0.5em}
One of the things that we are particularly interested in is adaptively estimating time-varying mass $m$, mass center $\p_c^i$ and inertia tensor $\mathcal{I}_i$ with $g_{ji}$ given since these parameters may vary in cooperative manipulation and transportation, e.g., by removing and adding some loads, and play a significant role for the control law design and system stability analysis. \par

Generally speaking, \cref{assump::1}-4 are common and reasonable for multi-robot coordination and robotic manipulation and some of them may be even further relaxed. The feasibility of measurements in \cref{assump::4} is analyzed in \cref{subsection::sensor}.
\subsection{Feasibility of Sensor Measurements}\label{subsection::sensor}
Body-fixed angular velocity $\bm{\omega}_i$ and proper acceleration $\overline{\a}_i$ are respectively measurable with gyroscope and accelerometer. The major concern is the body-fixed linear velocity $\v_i$, force $\bm{f}_i$ and torque $\bm{\tau}_i$ which may need a case by case discussion.\par
\subsubsection{body-fixed linear velocity $\v_i$} 
The body-fixed linear velocity can be calculated by the angular velocity of each joint and the twist of mobile base provided the manipulator Jacobian is given, all of which are either measurable with common sensors or explicitly computable. If the rigid body's orientation and spatial linear velocity is known, body-fixed linear velocity may be determined as well. As for planar cooperative manipulation, body-fixed linear velocity can be measured directly by a optical laser sensor \cite{wang2016kinematic}. 
\subsubsection{force $\bm{f}_i$ and torque $\bm{\tau}_i$}
If the Jacobian matrix is full row rank and the torques at each joint are known, then $\bm{f}_i$ and $\bm{\tau}_i$ are determined. Sometimes we may need some further assumptions, e.g., if all contacts are point contacts so that robots can only apply forces, then only force sensors should be enough.

\section{Observation Modelling}\label{section::measure}
In this section, we derive truly linear observation models for the identification problem formulated in \cref{section::problem}. It can be further shown that the state-dependent uncertainties of these observations can be exactly evaluated.
\subsection{Observation Model for $g_{ji}$}\label{subsection::pose}
It is known that sensor frames are fixed w.r.t. each other, then body-fixed angular and linear velocities $\bm{\omega}_i$, $\bm{v}_i$ and $\bm{\omega}_j$, $\bm{v}_j$ where $(i,\,j)\in E$ are related as
\begin{equation}\label{eq::Ad}
\bm{\xi}_j = \mathrm{Ad}_{g_{ji}}\bm{\xi}_i
\end{equation}
where $\bm{\xi}_i=\begin{bmatrix}
\bm{\omega}_i^T & \bm{v}_i^T
\end{bmatrix}^T$, $\bm{\xi}_j=\begin{bmatrix}
\bm{\omega}_j^T & \bm{v}_j^T
\end{bmatrix}^T$ and $\mathrm{Ad}_{g_{ji}}$ is the adjoint matrix defined by 
\begin{equation}
\nonumber
\mathrm{Ad}_{g_{ji}}=\begin{bmatrix}
R_{ji} & \bm{O}\\
\bm{t}_{ji}^\times R_{ji} & R_{ji}
\end{bmatrix}.
\end{equation}
\cref{{eq::Ad}} gives an observation with $\bm{\omega}_i$, $\bm{v}_i$, $\bm{\omega}_j$, $\bm{v}_j$ for the relative pose $g_{ji}$. It is difficult and even intractable to estimate $R_{ji}$ and $\p_{ji}$ directly on $SE(3)$ with \cref{eq::Ad} which is nonlinear and complicated.\par
Even though numbers of papers \cite{daniilidis1996dual,srivatsan2016estimating} have used dual quaternions to estimate elements of $SE(3)$, all of them rely on position measurements in $\R^3$ while for our problem only twist measurements in $\R^6$ are provided. The following proposition demonstrates how to do adjoint transformation with dual quaternions so that a linear observation model is developed to estimate $g_{ji}$.  
\begin{prop}\label{prop::1}
Given twist $\bm{\xi}=\begin{bmatrix}
\bm{\omega}\\
\v
\end{bmatrix}\in\R^6$ and unit dual quaternion $\widehat{\bm{x}}$ for rigid body transformation $g\in SE(3)$, then the resulting twist $\bm{\xi}'=\begin{bmatrix}
\bm{\omega}'\\
\v'
\end{bmatrix}\in\R^6$ by adjoint transformation of $g$ can be calculated by $\widehat{\bm{x}}$ as
\begin{equation}\label{eq::dquat_Ad}
\widehat{\bm{\xi}}' = \widehat{\bm{x}}\otimes\widehat{\bm{\xi}}\otimes\widehat{\bm{x}}^{2*}.
\end{equation}
where $\widehat{\bm{\xi}}=\widetilde{\bm{\omega}}+\epsilon\widetilde{\bm{v}}$.
\end{prop} 
\begin{proof}
It can be shown merely by calculation that
\begin{multline}\label{eq::qr_all}
\widehat{\bm{x}}\otimes\widehat{\bm{\xi}}\otimes\widehat{\bm{x}}^{2*}
=\widetilde{\q}_r\odot\widetilde{\bm{\omega}}\odot \widetilde{\q}_r^*+\epsilon\big(\widetilde{\q}_r\odot\widetilde{\bm{v}}\odot \widetilde{\q}_r^*+\\ \widetilde{\q}_d\odot\widetilde{\bm{\omega}}\odot \widetilde{\q}_r^*+\widetilde{\q}_r\odot\widetilde{\bm{\omega}}\odot \widetilde{\q}_d^*\big).
\end{multline}
Note that $\widetilde{R\bm{\omega}}=\widetilde{\q}_r\odot\widetilde{\bm{\omega}}\odot \widetilde{\q}_r^*$ and $\widetilde{R\bm{v}}=\widetilde{\q}_r\odot\widetilde{\bm{v}}\odot \widetilde{\q}_r^*$, then for $\widetilde{\q}_d\odot\widetilde{\bm{v}}\odot \widetilde{\q}_r^*+\widetilde{\q}_r\odot\widetilde{\bm{v}}\odot \widetilde{\q}_d^*$, a pure algebraic manipulation indicates
\begin{equation}\label{eq::qr_lin}
\begin{aligned}
&\widetilde{\q}_d\odot\widetilde{\bm{\omega}}\odot \widetilde{\q}_r^*+\widetilde{\q}_r\odot\widetilde{\bm{\omega}}\odot \widetilde{\q}_d^*\\
=&\dfrac{1}{2}\,\widetilde{\bm{t}}\odot\widetilde{\q}_r\odot\widetilde{\bm{\omega}}\odot \widetilde{\q}_r^* + \dfrac{1}{2}\,\widetilde{\q}_r\odot\widetilde{\bm{\omega}}\odot \widetilde{\q}_r^*\odot \widetilde{\bm{t}}^*\\
= & \dfrac{1}{2}\widetilde{\bm{t}\times R\bm{\omega}} - \dfrac{1}{2}\widetilde{ R\bm{\omega}\times\bm{t}} \\
= & \widetilde{\bm{t}\times R\bm{\omega}}.
\end{aligned}
\end{equation}
Substitute \cref{eq::qr_lin} back to \cref{eq::qr_all}, it can be shown that \cref{eq::dquat_Ad} holds, which completes the proof.
\end{proof}
As a result of \cref{prop::1}, \cref{eq::Ad} is equivalent to 
\begin{equation}\label{eq::mea1}
\widehat{\bm{\xi}}_j =\widehat{\bm{x}}_{ji}\otimes \widehat{\bm{\xi}}_i\otimes\widehat{\bm{x}}_{ji}^{2*}.
\end{equation}
Since $\widehat{\bm{x}}_{ji}^{2*}\otimes\widehat{\bm{x}}_{ji}=1+\epsilon\widetilde{\0}$, \cref{eq::mea1} can be written as $\widehat{\bm{\xi}}_j\otimes\widehat{\bm{x}}_{ji} =\widehat{\bm{x}}_{ji}\otimes \widehat{\bm{\xi}}_i$, and according to \cref{eq::duatLR2}, we further have $$\widehat{\bm{\xi}}_{j}^+\cdot\widehat{\bm{x}}_{ji}= \widehat{\bm{\xi}}_{i}^-\cdot\widehat{\bm{x}}_{ji}.$$ Next, simplify the equation above with \cref{eq::quatLR,eq::duatLR}, the final result is
\begin{equation}\label{eq::qr_ob}
\begin{bmatrix}
H_q(\bm{\omega}_i,\,\bm{\omega}_j) & \bm{O}\\
H_q(\bm{v}_i,\,\bm{v}_j) & H_q(\bm{\omega}_i,\,\bm{\omega}_j)
\end{bmatrix}\cdot
\begin{bmatrix}
\widetilde{\q}^r_{ji}\\
\widetilde{\q}^d_{ji}
\end{bmatrix}=\0
\end{equation}
where $H_q:\R^3\times \R^3 \longrightarrow \R^{4\times 4}$ is defined as
\begin{equation}
\nonumber
H_q(\a,\,\b) =\begin{bmatrix}
0 &(\a-\b)^T\\
\b-\a & (\a+\b)^\times
\end{bmatrix}.
\end{equation}\par
In practice, $\bm{\omega}_i$, $\bm{v}_i$, $\bm{\omega}_j$, $\bm{v}_j$ are noisy measurements and we can construct a pseudo-observation model from \cref{eq::qr_ob} 
\begin{equation}
\bm{y}=\begin{bmatrix}
H_q(\bm{\omega}_i,\,\bm{\omega}_j) & \bm{O}\\
H_q(\bm{v}_i,\,\bm{v}_j) & H_q(\bm{\omega}_i,\,\bm{\omega}_j)
\end{bmatrix}\cdot
\begin{bmatrix}
\widetilde{\q}^r_{ji}\\
\widetilde{\q}^d_{ji}
\end{bmatrix}
\end{equation}
and enforce the pseudo-observation $\bm{y}=\0$ just like \cite{srivatsan2016estimating,choukroun2006novel} with position measurements. Thus a linear observation model $\y = H\cdot \widehat{\bm{x}}$ is developed to estimate relative rigid body transformation matrix $g_{ji}$ with dual quaternions. A specific filter is developed to solve the dual-quaternion-based pose estimation problem in \cref{subsection::pose_est}.
\subsection{Observation Model for $\bm{\alpha}_i$}
The body-fixed angular acceleration $\bm{\alpha}(t)$ can be got by differentiating $\bm{\omega}(t)$ as
\begin{equation}\label{eq::diff_w}
\bm{\alpha}(t)=\frac{\bm{\omega}(t+\Delta t)-\bm{\omega}(t-\Delta t)}{2\Delta t}.
\end{equation} 
However, the angular acceleration $\bm{\alpha}(t)$ evaluated by \cref{eq::diff_w} is noisy and an observation model may be needed. Note that the relative pose $g_{ji}$ is time-invariant, then for angular and proper linear accelerations $\bm{\alpha}_i$, $\overline{\a}_i$ and $\bm{\alpha}_j$, $\overline{\a}_j$ we have
$$\begin{bmatrix}
\bm{\alpha}_j\\
\overline{\a}_j
\end{bmatrix} = \Ad_{g_{ji}} \begin{bmatrix}
\bm{\alpha}_i\\
\overline{\a}_i
\end{bmatrix}.$$ 
Provided $g_{ji}=(R_{ji},\,\bm{t}_{ji})\in SE(3)$, $\overline{\a}_i$ and $\overline{\a}_j$ are given, a pseudo-observation model for $\bm{\alpha}_i$ can be derived
\begin{equation}
\y = H_{\bm{\alpha}}\bm{\alpha}_i
\end{equation}
where $H_{\bm{\alpha}}=\bm{t}_{ji}^\times R_{ji}$ and the enforced pseudo-observation is $\y=\overline{\a}_j-R_{ji}\overline{\a}_i$. In addition, from
\begin{equation}\label{eq::con_alpha}
\bm{\alpha}_j = R_{ji}\bm{\alpha}_i,
\end{equation} 
the estimation can be further improved through making consensus over $\bm{\alpha}_i$ for all $i\in V$. 
\subsection{Observation Model for $\p_c^i$}\label{subsection::pc}
It is known that the dynamics on $SE(3)$ is
\begin{subequations}\label{eq::se3_dyn}
\begin{equation}\label{eq::se3_dyn_1}
\mathcal{I}\dot{\bm{\omega}} = \bm{T}- \bm{\omega}\times \mathcal{I}\bm{\omega},
\end{equation}
\begin{equation}\label{eq::se3_dyn_2}
\dot{\v}= \dfrac{\bm{F}}{m}+ R^T\bm{g}-\bm{\omega}\times\v
\end{equation}
\end{subequations}
for which the body-fixed frame origin is at the mass center of the rigid body. By \cref{eq::prop_acc} we may rewrite \cref{eq::se3_dyn} as 
\begin{subequations}\label{eq::dyn}
\begin{equation}\label{eq::dyn_1}
\mathcal{I}\dot{\bm{\omega}} = \bm{T}- \bm{\omega}\times \mathcal{I}\bm{\omega},
\end{equation}
\begin{equation}\label{eq::se3_dyn_2}
\overline{\a}= \dfrac{\bm{F}}{m}-\bm{\omega}\times\v
\end{equation}
\end{subequations}
where $\overline{\a}=\dot{\v}-R^T\bm{g}$ is the proper acceleration of the mass center.
\par
Let $\p_c^i$ be the observation of the mass center in $\mathcal{S}_i$. Then for each robot $i$ a local frame $\mathcal{O}_i$ is assigned at the mass center so that the rigid body transformation matrix $g_i$ from $\mathcal{S}_i$ to $\mathcal{O}_i$ is
$$g_i=\begin{bmatrix}
\mathrm{I} & -\p_c^i\\
\0 & 1
\end{bmatrix}. $$
Given body-fixed angular and linear velocities $\bm{\alpha}_i$ and $\v_i$, angular and linear proper accelerations $\bm{\alpha}_i$ and $\overline{\a}_i$, \cref{eq::dyn} indicates that
\begin{subequations}\label{eq::dyn_ob}
\begin{equation}\label{eq::dyn_ob1}
\mathcal{I}_i\bm{\alpha}_i = \bm{F}_i\times\p_c^i  +\bm{T}_i- \bm{\omega}_i\times \mathcal{I}_i\bm{\omega}_i,
\end{equation}
\begin{equation}\label{eq::dyn_ob2}
\overline{\a}_i+\bm{\alpha}_i\times\p_{c}^i= \dfrac{\bm{F}_i}{m}-\bm{\omega}_i\times(\bm{\omega}_i\times \p_c^i+\v_i)
\end{equation}
\end{subequations}
where $\overline{\a}_i$ is measured by a accelerometer and the total force $\bm{F}_i$ and torque $\bm{T}_i$ applied by all robots are calculated by
\begin{equation}\label{eq::wrench}
\begin{bmatrix}
\bm{T}_i\\
\bm{F}_i
\end{bmatrix}=\sum\limits_{j\in V} \Ad_{g_{ji}}^T\begin{bmatrix}
\bm{\tau}_j\\
\bm{f}_j
\end{bmatrix}.
\end{equation}
which is shown to be computable by making consensus in different coordinates in \cref{subsection::consensus}.\par 
\cref{eq::dyn_ob2} is equivalent to  
\begin{equation}\label{eq::ob_m}
(\bm{\alpha}_i^\times+{\bm{\omega}_i^\times}^2)\cdot\p_c^i+\bm{\omega}_i\times\v_i+\overline{\a}_i = \frac{\bm{F}_i}{m}.
\end{equation}
Note that the mass $m$ remains unknown, however, if $\bm{F}_i\neq \0$, we may let $\bm{F}_i^\perp\in \R^{2\times 3}$ be the row-orthogonal matrix such that 
$$\bm{F}_i^\perp\cdot \bm{F}_i =\0, $$
i.e., columns of ${\bm{F}_i^\perp}^T$ spans the null space of $\bm{F}_i$. Multiply $\bm{F}_i^\perp$ on both sides of \cref{eq::ob_m}, the resulting equation is
\begin{equation}\label{eq::ob_mi2}
\bm{F}_i^\perp(\bm{\alpha}_i^\times+{\bm{\omega}_i^\times}^2)\cdot\p_c^i+\bm{F}_i^\perp\cdot(\bm{\omega}_i\times\v_i+\overline{\a}_i)=\0.
\end{equation}
Thus an observation model for $\p_c^i$ is obtained by \cref{{eq::ob_mi2}}
\begin{equation}\label{eq::ob_pci}
\y_{\p_c}=H_{\p_c}\cdot\p_c^i
\end{equation}
where $H_{\p_c}=\bm{F}_i^\perp(\bm{\alpha}_i^\times+{\bm{\omega}_i^\times}^2)$ and $\y_{p_c}=-\bm{F}_i^\perp\cdot(\bm{\omega}_i\times\v_i+\overline{\a}_i)$.
\subsection{Observation Model for $\mathcal{I}_i$}\label{subsection::I}
Suppose the mass center $\p_c^i$ is known, then an observation model for $\mathcal{I}_i$ can be derived by \cref{eq::dyn_ob1} so that
\begin{equation}\label{eq::ob_mi1}
\y_{\mathcal{I}}=H_{\mathcal{I}}\cdot\mathcal{I}_i^S
\end{equation} 
where $\mathcal{I}_i^S=\begin{bmatrix}
\mathcal{I}_{xx}^i & \mathcal{I}_{yy}^i & \mathcal{I}_{zz}^i & \mathcal{I}_{xy}^i & \mathcal{I}_{xz}^i & \mathcal{I}_{yz}^i
\end{bmatrix}^T$,
$$
\begin{aligned}
\!\!&H_{\mathcal{I}}=\\
&\!\!\begin{bmatrix}
\!\!\alpha_1 \!\!\!& \!\!\! -\omega_2 \omega_3 \!\!\!& \!\!\! \omega_2 \omega_3 \!\!& \!\! \alpha_2-\omega_1 \omega_3 \!\!& \!\! \alpha_3+\omega_1 \omega_2 \!\!& \!\! \omega_2^2-\omega_3^2\\
\!\!\omega_1 \omega_3 \!\!\!& \!\!\! \alpha_2 \!\!\!& \!\!\! -\omega_1 \omega_3 \!\!& \!\! \alpha_1+\omega_2 \omega_3 \!\!& \!\! \omega_3^2-\omega_1^2 \!\!& \!\! \alpha_3-\omega_1 \omega_2\\
\!\!-\omega_1 \omega_2 \!\!\!& \!\!\! \omega_1 \omega_2 \!\!& \!\! \alpha_3 \!\!& \!\!
\omega_1^2-\omega_2^2 \!\!& \!\! \alpha_1-\omega_2\omega_3 \!\!& \!\! \alpha_2+\omega_1\omega_3\!\!
\end{bmatrix}\!\!.
\end{aligned}
$$
and $\y_{\mathcal{I}}=\bm{F}_i\times\p_c^i  +\bm{T}_i$.
\subsection{Observation Model for $m$}\label{subsection::m}
Given $\p_c^i$ and these measurements in \cref{assump::4}, the pseudo-observation model for mass $m$ is trivial from \cref{eq::dyn_ob2}
\begin{equation}\label{eq::ob_mm}
\y_m = H_m  \!\cdot \!m
\end{equation}
where $H_m=\overline{\a}_i+\bm{\alpha}_i\times\p_{c}^i+\bm{\omega}_i\times(\bm{\omega}_i\times \p_c^i+\v_i)$ and
$\y_m=\bm{F}_i.$
\subsection{Observation Model for $\p_c^i$, $\mathcal{I}_i$ and $m$}
We have developed individual observation models for mass center $\p_c^i$, mass $m$ and inertia tensor $\mathcal{I}_i$ among which the estimations of $m$ and $\mathcal{I}_i$ depend on that of $\p_c^i$. In general, we may prefer to estimate $\p_c^i$, $m$ and $\mathcal{I}_i$ at the same time rather than separately since the former should be more robust and more accurate. Note that $\p_c^i$, $m$ and $\mathcal{I}_i$ are all constants, we may observe all of them just in one model.\par
For $\p_c^i$ and $\mathcal{I}_i$, by \cref{eq::dyn_ob1,eq::ob_mi1} we have 
\begin{equation}\label{eq::oba1}
\begin{bmatrix}
-\bm{F}_i^\times & H_{\mathcal{I}}
\end{bmatrix}
\begin{bmatrix}
\p_c^i\\
\mathcal{I}^S
\end{bmatrix}=\bm{T}_i.
\end{equation}
Besides if $m$ is replaced by $\frac{1}{m}$ as the estimated unknown, an observation model for $\p_c^i$ and $\frac{1}{m}$ from \cref{eq::dyn_ob2} is
\begin{equation}\label{eq::oba2}
\begin{bmatrix}
-\bm{\alpha}_i^\times-{\bm{\omega}_i^\times}^2 &  \bm{F}_i
\end{bmatrix}
\begin{bmatrix}
\p_{c}^i\\
\frac{1}{m}
\end{bmatrix}
= \bm{\omega}_i\times\v_i+\overline{\a}_i.
\end{equation}
According to \cref{eq::oba1,eq::oba2,eq::ob_mi2}, we may derive a model to simultaneously observe $\p_c^i$, $m$ and $\mathcal{I}$ as 
\begin{equation}\label{eq::ob_all}
\y_D = H_D \begin{bmatrix}
\p_c^i\\
\mathcal{I}^S\\
\frac{1}{m}
\end{bmatrix}
\end{equation}
where
\begin{equation}
\nonumber
H_D = \begin{bmatrix}
\bm{F}_i^\perp(\bm{\alpha}_i^\times+{\bm{\omega}_i^\times}^2) & \bm{O} &  \0\\
-\bm{F}_i^\times & H_{\mathcal{I}} &  \0\\
-\bm{\alpha}_i^\times-{\bm{\omega}_i^\times}^2 &\bm{O} &  \bm{F}_i
\end{bmatrix}
\end{equation}
and 
\begin{equation}
\nonumber
\y_D=\begin{bmatrix}
-\bm{F}_i^\perp\cdot(\bm{\omega}_i\times\v_i+\overline{\a}_i)\\
\bm{T}_i\\
\bm{\omega}_i\times\v_i+\overline{\a}_i
\end{bmatrix}.
\end{equation}
In this paper, \cref{eq::ob_all} is used in \cref{section::num} for numerical simulation. Observation models of \cref{eq::ob_pci,eq::ob_mm,eq::ob_mi1} can be used if some of $\p_c^i$, $m$ and $\mathcal{I}_i$ are assumed to be known.\par

The consensus over each inertia tensor $\mathcal{I}_i$ and mass center $\p_c^i$ estimated in $\mathcal{S}_i$ can be made as
\begin{equation}\label{eq::con_I}
\mathcal{I}_j=R_{ji}^T\cdot\mathcal{I}_i \cdot R_{ji}
\end{equation}
and
\begin{equation}\label{eq::con_pc}
\begin{bmatrix}
\p_c^j\\
1
\end{bmatrix}
 = 
\begin{bmatrix}
 R_{ji} & \bm{t}_{ji}\\
 \0 & 1
\end{bmatrix} 
\begin{bmatrix}
\p_c^i\\
1
\end{bmatrix}
\end{equation}
by dynamic consensus in different coordinates to increase identification belief.
\section{Filtering}\label{section::filter}
In \cref{section::measure}, we construct truly linear models in forms of $\y=H\x$ for each sub-problem and all unknowns to be estimated are time-invariant constants.\par 
A general recursive-least-square-like filter may be formulated as \cref{eq::filter} for estimation with linear measurements
\begin{subequations}\label{eq::filter}
\vspace{-1em}
\begin{multline}\label{eq::filter_1}
\x_{k+1} = \arg \min_{\x}\Big\{ \dfrac{1}{2}(\x-\x_k)^T P_k^{-1} (\x-\x_k)+\\ \dfrac{1}{2}(\y_k-H_k\x_k)^TR_k^{-1}(\y_k-H_k\x_k)\Big\}
\end{multline}
and
\begin{equation}
P_{k+1} = \lambda\cdot(P_k^{-1} +H_k^T R_{k}^{-1} H_k)^{-1}
\end{equation}
\end{subequations}
where $P_k$ and $R_k$ are the covariance matrices of $\x_k$ and $H_k\x_k-\y_k$ while $\lambda\geq 1$ is the forgetting factor. Besides if $\lambda=1$ and $\x_k$ is not constrained, then \cref{eq::filter} is just the correction step in Kalman filtering\cite{sorenson1970least}.\par
In this section, we will discuss how to exactly solve the formulated identification problem for cooperative manipulation with appropriate distributed filtering techniques.
\subsection{State-dependent Uncertainties Evaluation}
All the observation models constructed in \cref{section::measure} are pseudo whose either observation matrix $H_k$ or observation $y_k$ or both depend on the unknown $\x_k$ and the noisy measurements. Though it remains numerically feasible by just assuming the covariance matrix $R_k$ is constant which is simplified to a basic least square estimation, a explicit evaluation of the covariance $R_k$ for $H\x_k-\y_k$ is still preferable. This is possible with suitable independence assumptions as the following proposition indicates \cite{choukroun2006novel,srivatsan2016estimating}.
\begin{prop}\label{prop::uncertain}
Let us consider $\bm{b}\in\R^m$ and $\bm{c}\in \R^m$ which are sequences with zero mean. Let $\bm{h}\in\R^n$, $\x\in\R^n$ and a linear
matrix function $G:\R^l\longrightarrow \R^{n\times m}$, such that $\bm{y} = G(\x)\bm{b}+\bm{c}$. Assume that $\x$, $\bm{b}$ and $\bm{c}$ are independent. Then $\Sigma^{\bm{y}}$
$$\Sigma^{\bm{y}} = G(\x)\Sigma^{\b}G^T(\x)+\bm{N}(\Sigma^{\b}\circledast \Sigma^{\x})\bm{N}^T+\Sigma^{\bm{c}}$$
where $\circledast$ is the Kronecker product, $\Sigma^{(\cdot)}$ is the uncertainty associated with $\{\cdot\}$ and $\bm{N}\in \R^{n\times lm}$ is defined as follows
$$\bm{N} \triangleq \begin{bmatrix}
\bm{G}_1 & \bm{G}_2 &\cdots& \bm{G}_m,
\end{bmatrix} $$
$\bm{G}\in\R^{n\times l}$ is obtained from the following identity
$$ \bm{G}_i\x=\bm{G}(\x)\bm{e}_i$$
where $\bm{e}_i$ is column $i$ of the identity matrix of $\R^{m\times m}$.
\end{prop}
\cref{prop::uncertain} enables us to evaluate the covariance $R_k$ for $H\x_k-y_k$ only with some independence assumptions of the unknowns and measurements. Due to space limitation, we may not demonstrate this in detail. Readers may refer to \cite{choukroun2006novel,srivatsan2016estimating} for some examples to implement this proposition.
\subsection{Dynamic Consensus in Different Coordinates}\label{subsection::consensus}
For this identification problem, dynamic consensus is needed to compute total wrench and improve the belief over different estimations of unknowns that are essentially the same. Even though each individual robot makes estimations and apply forces and torques in its local reference frame, it is still likely to make consensus in different coordinates by communicating with its neighbours as long as the network is connected.
\begin{prop} \label{prop_3}
Suppose $G=(V,\,E)$ is an undirected $n$-node graph and each node $i$ has initial value $\x_i(0)$. Let $\{A_{ji}|i,\,j\in V\}$ be a set of time-invariant linear transformations such that $A_{ii}=\I$ and $A_{ij}=A_{ik}\cdot A_{kj}$ for any $i,\,j,\,k\in V$, then for the following dynamical system
\begin{equation}\label{eq::con_dyn_11}
\dot{\x}_i(t)=  \sum\limits_{j\in N_i}\Big[ A_{ij}\x_j(t)-\x_i(t)\Big]
\end{equation}
we have
\begin{equation}\label{eq::result_11}
\x_i(t)\longrightarrow \dfrac{1}{n}\sum\limits_{j\in V}A_{ij}\x_j(0)
\end{equation}
and
$$A_{ji}\x_i(t)\longrightarrow \x_{j}(t) $$
as $t\rightarrow \infty$ if $G$ is connected. 
\end{prop}
\begin{proof}
Let $\y_i(t) = A_{ki}\x_i(t)$. As $t\rightarrow \infty$, we have 
\begin{equation}\label{eq::result_22}
\y_i(t) \rightarrow \dfrac{1}{n}\sum\limits_{j\in V} \y_j(0)
\end{equation}
with the following dynamics
\begin{equation}\label{eq::con_dyn2}
\dot{\y}_i(t) =\sum\limits_{j\in N_i}\Big[\y_{j}(t)-\y_i(t)\Big]. 
\end{equation}
Note $A_{ij}=A_{ik}A_{kj}$ and then multiply $A_{ik}$ on both sides of \cref{eq::result_22,eq::con_dyn2} for each $i$, the resulting equations are \cref{eq::result_11,eq::con_dyn_11},
which completes the proof.
\end{proof}
\cref{prop_3} can be further generalized to other cases as these in \cite{spanos2005dynamic} so that distributed filtering techniques may be used \cite{olfati2007distributed}. As for consensus in our problem, suitable linear transformations $\{A_{ij}\}$ can be $g_{ji}$, $R_{ji}$, $\Ad_{g_{ji}}$ etc. as those shown in \cref{eq::con_alpha,eq::wrench,eq::con_I,eq::con_pc}. \par
\subsection{Filtering on Dual Quaternions}
Filtering on quaternions and dual quaternions have been studied in \cite{kraft2003quaternion,choukroun2006novel,srivatsan2016estimating,shuster1993quaternion}. One of critical problems for quaternion and dual quaternion filtering is how to fulfil the unit requirements. Popular methods include regarding the constraint as an extra pseudo-observation \cite{deutschmann1992quaternion}, substituting the constraint to observation \cite{shuster1993quaternion}, or normalizing the filtering result at the end of each step \cite{srivatsan2016estimating}, which often either result in nonlinear observation models or converge to a local minima. In our specific problem, however, it can be shown that we may update the dual quaternion estimation $\widetilde{\q}_{ji}^r$ and $\widetilde{\q}_{ji}^d$ while simultaneously preserving the linearity of observation and satisfying the unit requirement without normalization. This method used for dual quaternion estimation is based on the fact that the relative pose $g_{ji}$ to be estimated is a constant. For brevity we have suppressed subscript ${ji}$ in Eqs (\ref{eq::quat_rot_filter})-(\ref{eq::quat_P_filter}).\par
For the rotational part $\widetilde{\q}^r_k$, since the pseudo-observation is $H_k\widetilde{\q}_k^r=\0$, \cref{eq::filter} may be reformulated as
\begin{subequations}\label{eq::quat_rot_filter}
\begin{equation}
P_k^{-1}=P_{k-1}^{-1}+H_k^T R_k^{-1}H_k,
\end{equation}
\begin{equation}\label{eq::quat_filter_1}
\widetilde{\q}^r_k = \arg\min_{\substack{\|\widetilde{\q}\|=1\\ q_0\geq 0}}\frac{1}{2}\widetilde{\q}^TP_k^{-1}\widetilde{\q}.
\end{equation}
\end{subequations}
in which \cref{eq::quat_filter_1} is equivalent to determining the minimal eigenvalue $\lambda_{\min}$ of $P_k^{-1}$ and can be exactly solved through eigenvalue decomposition. In cases that the multiplicity of $\lambda_{\min}$ is greater than 1, which may sometimes happen if there are not enough observations, we may determine the estimation by minimizing $\|\widetilde{\q}_k^r-\widetilde{\q}_{k-1}^r\|_{P^{-1}_{k-1}}$ among possible choices of $\widetilde{\q}_k^r$.\par
If $\widetilde{\q}^r_k$ is given, the estimation of $\widetilde{\q}_k^d$ is determined by 
\begin{equation}\label{eq::quat_lin_filter}
\widetilde{\q}_k^d = \arg\min_{\widetilde{\q}^T\cdot\widetilde{\q}_k^r=0}\frac{1}{2}\sum\limits_{l=1}^k \left\|H_l^{\bm{\omega}}\widetilde{\q}+H_l^{\v}\widetilde{\q}_k^r\right\|_{R_k^{-1}}^2
\end{equation}
where $H_l^{\bm{\omega}}$ and $H_l^{\v}$ are defined by \cref{eq::qr_ob} and the solution to which is 
\begin{equation}\label{eq::quat_P_filter}
\begin{bmatrix}
\widetilde{\q}_k^d\\
\mu
\end{bmatrix}=\begin{bmatrix}
P_k^{-1} & \widetilde{\q}_k^r\\
{{\widetilde{\q}_k^r}}{}^T & 0
\end{bmatrix}^{-1}
\begin{bmatrix}
S_k\cdot \widetilde{\q}_k^r\\
0
\end{bmatrix}
\end{equation}
in which $\mu$ is the Lagrangian multiplier and $P_k^{-1}$ and $S_k$ are recursively updated by
$$P_k^{-1}=P_{k-1}^{-1}+{H_k^{\bm{\omega}}}^T R_k^{-1}{H_k^{\bm{\omega}}}=\sum\limits_{l=1}^k{H_l^{\bm{\omega}}}^T R_l^{-1}{H_l^{\bm{\omega}}},$$
$$S_k = S_{k-1}-{H_l^{\bm{\omega}}}^T R_l^{-1}H_k^{\v}=-\sum\limits_{l=1}^k{H_l^{\bm{\omega}}}^T R_l^{-1}H_k^{\v}.$$
In this way, a recursive filter to estimate the dual quaternion $\widehat{\x}_{ji}=(\widetilde{\q}_{ji}^r,\,\widetilde{\q}_{ji}^d)$ is developed for estimating relative $g_{ji}\in SE(3)$. Note that with dual quaternion and eigenvalue decomposition, the resulting filter is linear and gives exact optimal solution to the generally nonlinear and nonconvex pose estimation problem on $SE(3)$.
\begin{figure}[t]
\centering
\begin{tabular}{cc}
  \vspace{-0.5em}
  \subfloat[][]{\includegraphics[trim =6mm 4mm 12mm 6mm,width=0.35\textwidth]{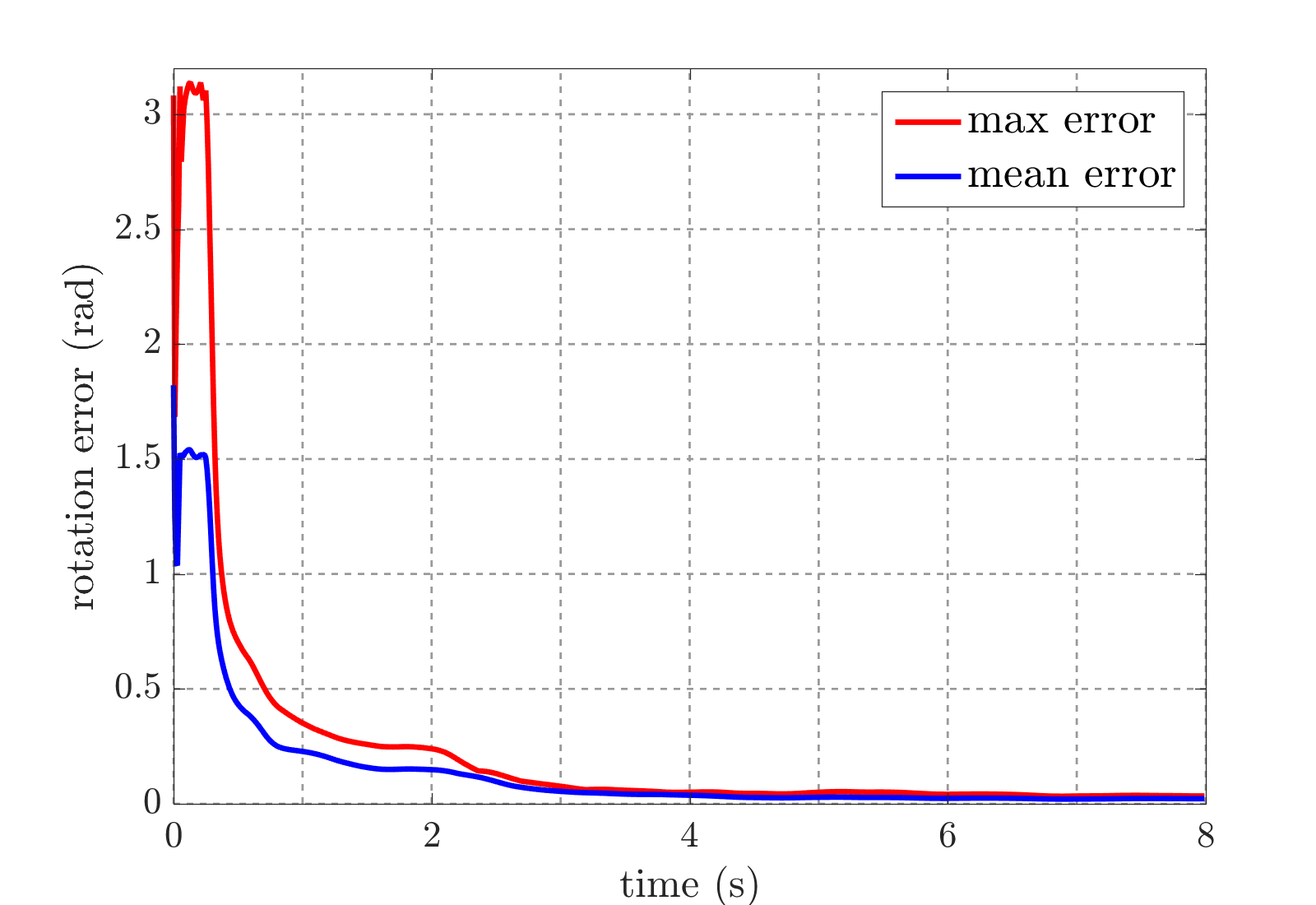}\label{fig::g_R}} \\
  \subfloat[][]{\includegraphics[trim =4mm 4mm 8mm 6mm,width=0.35\textwidth]{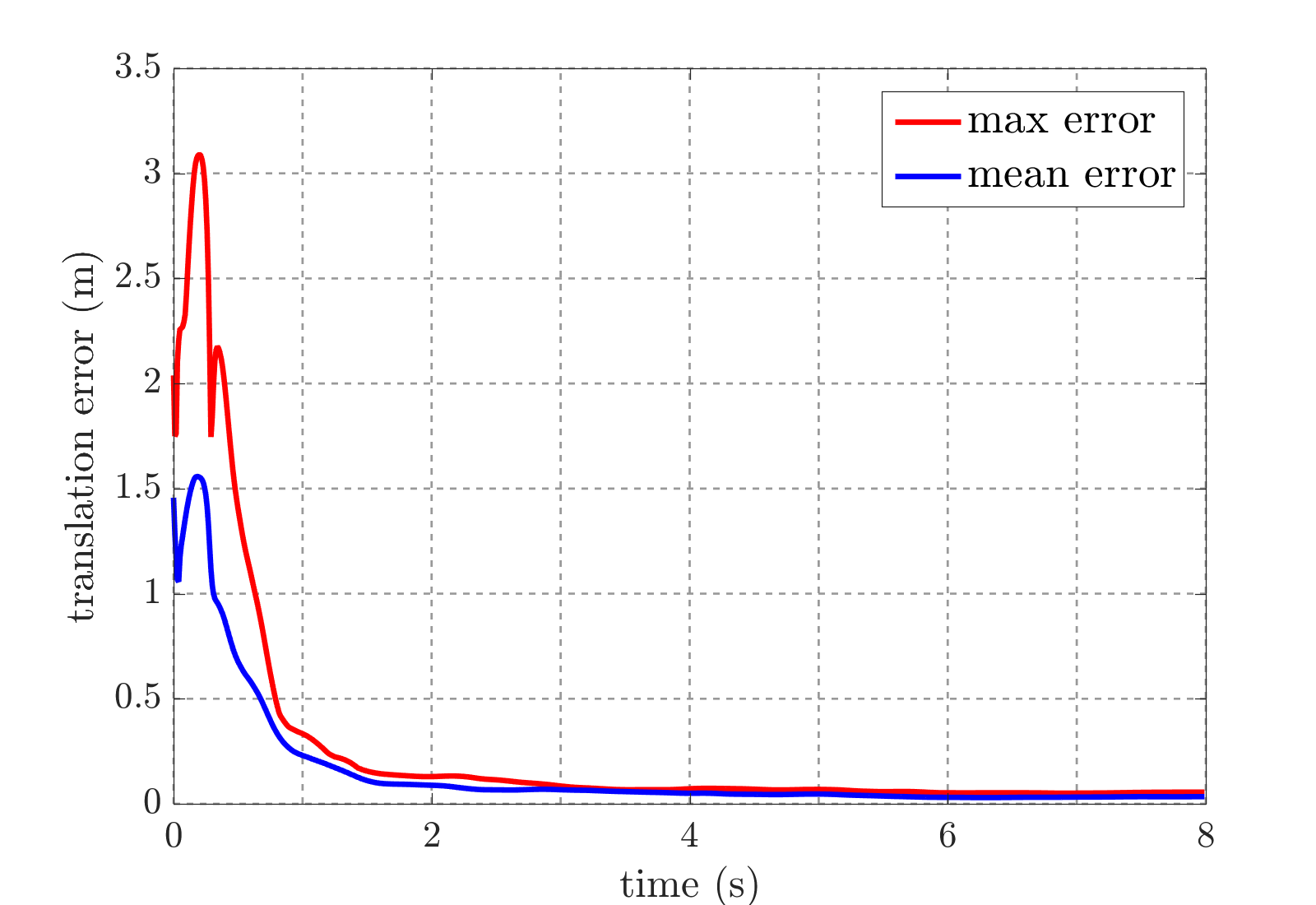}\label{fig::g_t}} 
\end{tabular}
\caption{Identification results for relative pose $g_{ji}$ in which (a) is the estimation error for rotation and (b) for translation. At $t=8$s, the mean error for rotation is $0.036\,\mathrm{rad}$ and mean error for translation is $0.032\;\mathrm{m}$. } 
\label{fig::err_g} 
\end{figure}
\begin{figure*}[!htbp]
\centering
\begin{tabular}{ccc}
  \subfloat[][]{\includegraphics[trim =23mm 2mm 5mm 2mm,width=0.32\textwidth]{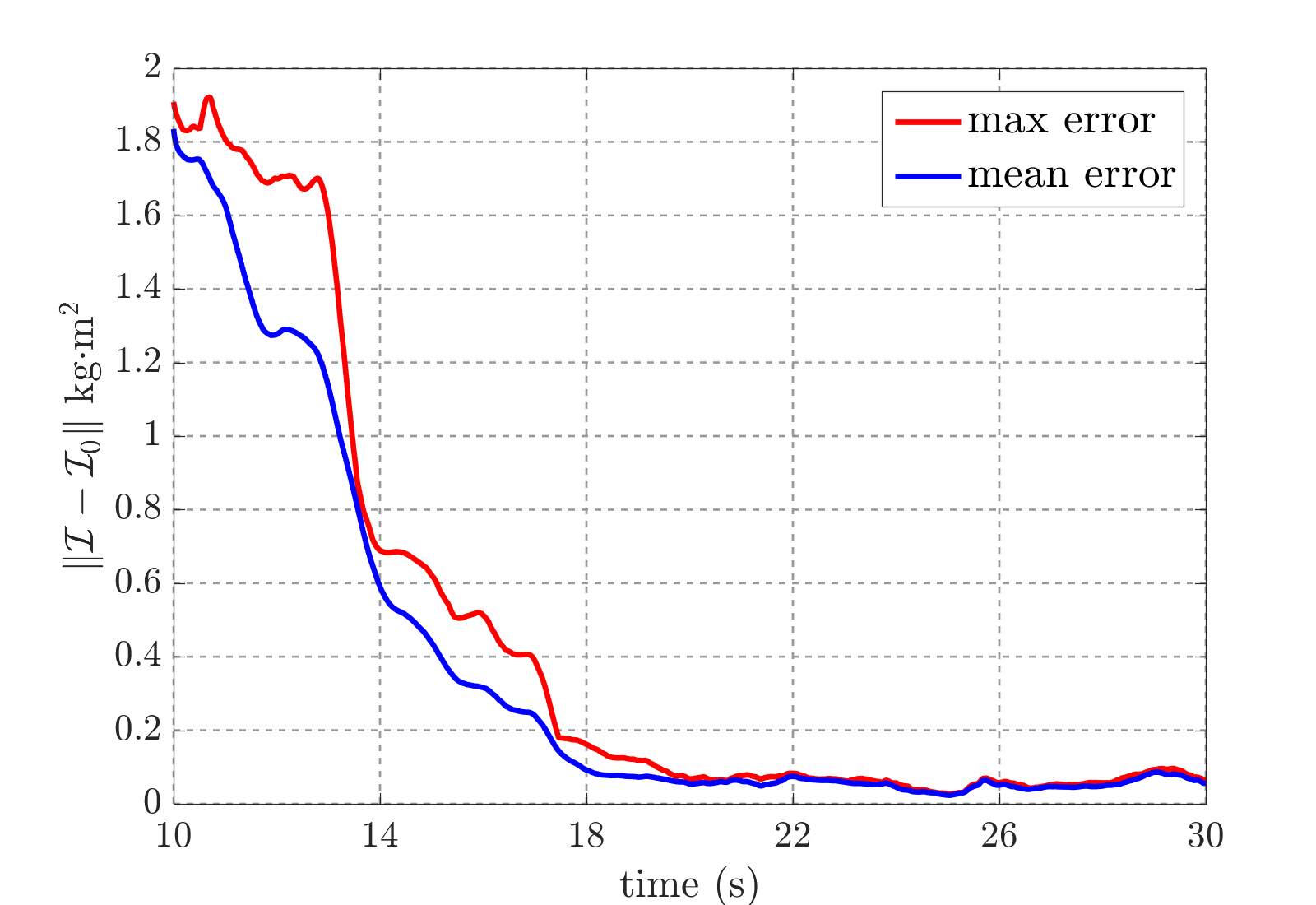}}\label{fig::J} &
  \subfloat[][]{\includegraphics[trim =23mm 2mm 5mm 2mm,width=0.32\textwidth]{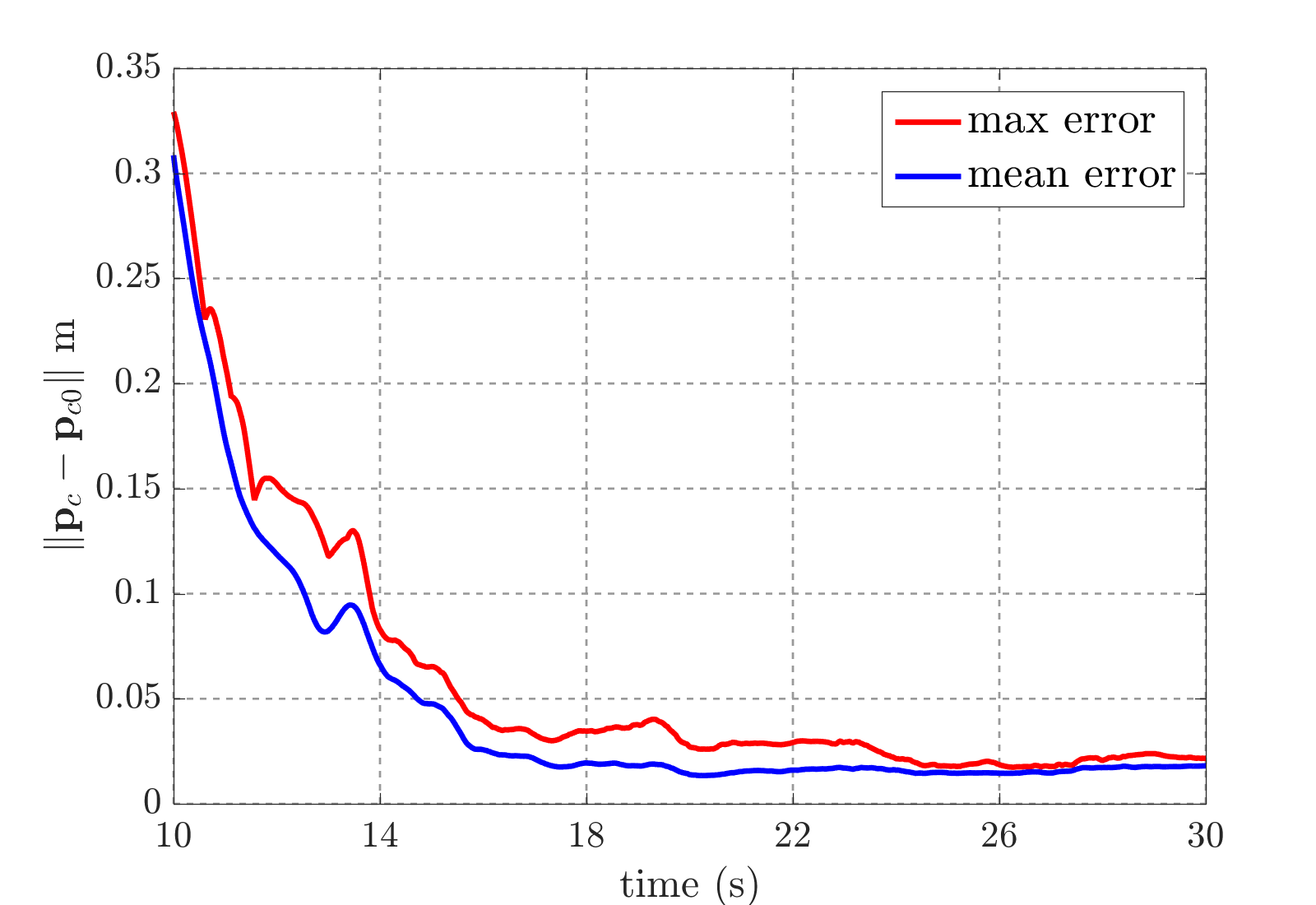}}\label{fig::pc} &
  \subfloat[][]{\includegraphics[trim =23mm 2mm 5mm 2mm,width=0.32\textwidth]{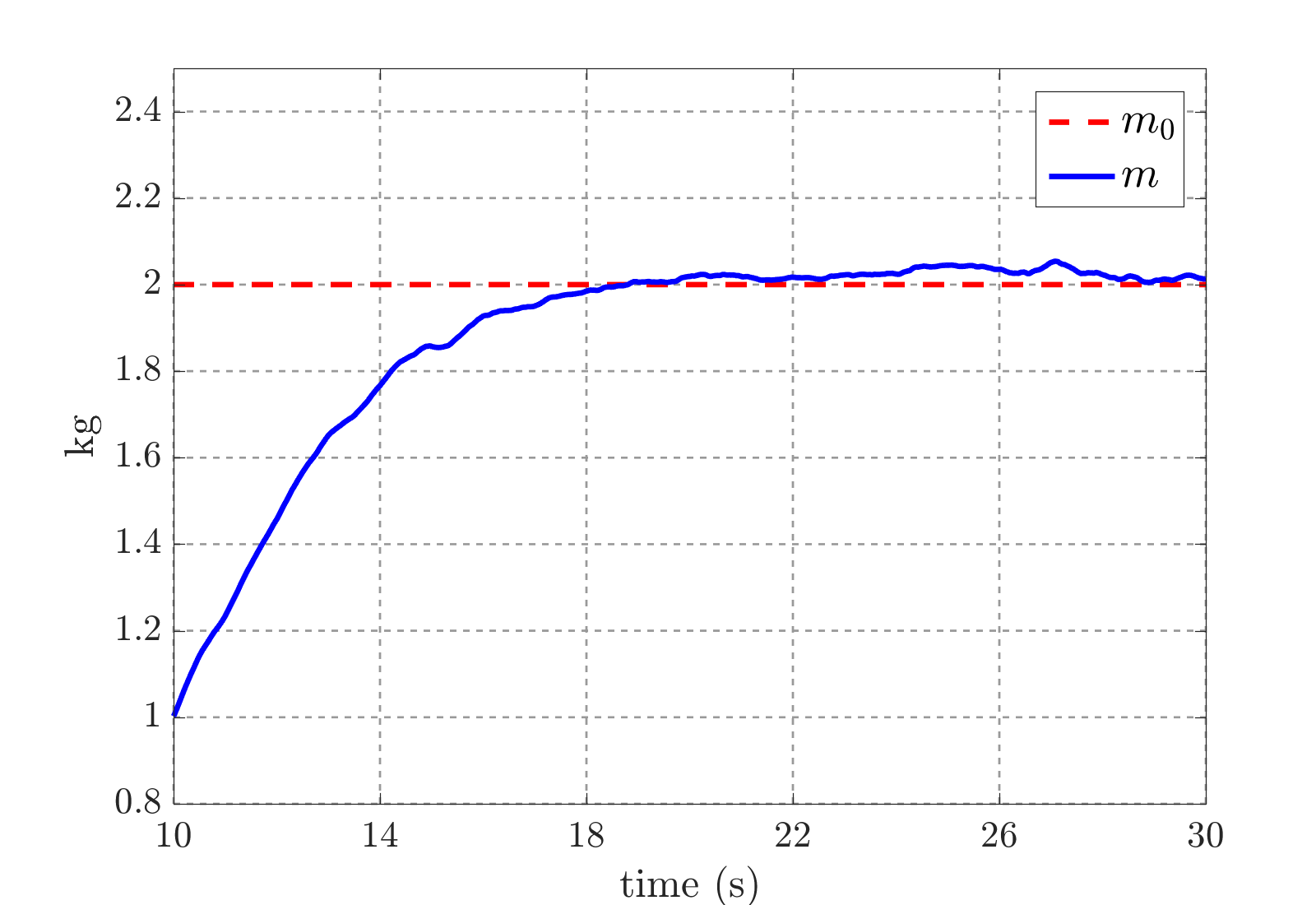}}\label{fig::m}
\end{tabular}
\caption{Estimation inertia tensor $\mathcal{I}$, mass center ${\p_c}$ and mass $m$. (a) is $\mathcal{I}$, (b) is ${\p_c}$ and (c) is $m$. At $t=20$s, the mean estimation error for $\mathcal{I}$ is $0.075\,\mathrm{kg}\cdot\mathrm{m}^2$ and for $\p_c$ is $0.015\,\mathrm{m}$ and the estimated $\hat{m}=2.03\,\mathrm{kg}$ while the true $m_0=2\,\mathrm{kg}$.}
\label{fig::inertia}  
\vspace{-1.5em}
\end{figure*}
\begin{figure*}[!htbp]
\centering
\begin{tabular}{ccc}
  \subfloat[][]{\includegraphics[trim =23mm 2mm 5mm 2mm,width=0.32\textwidth]{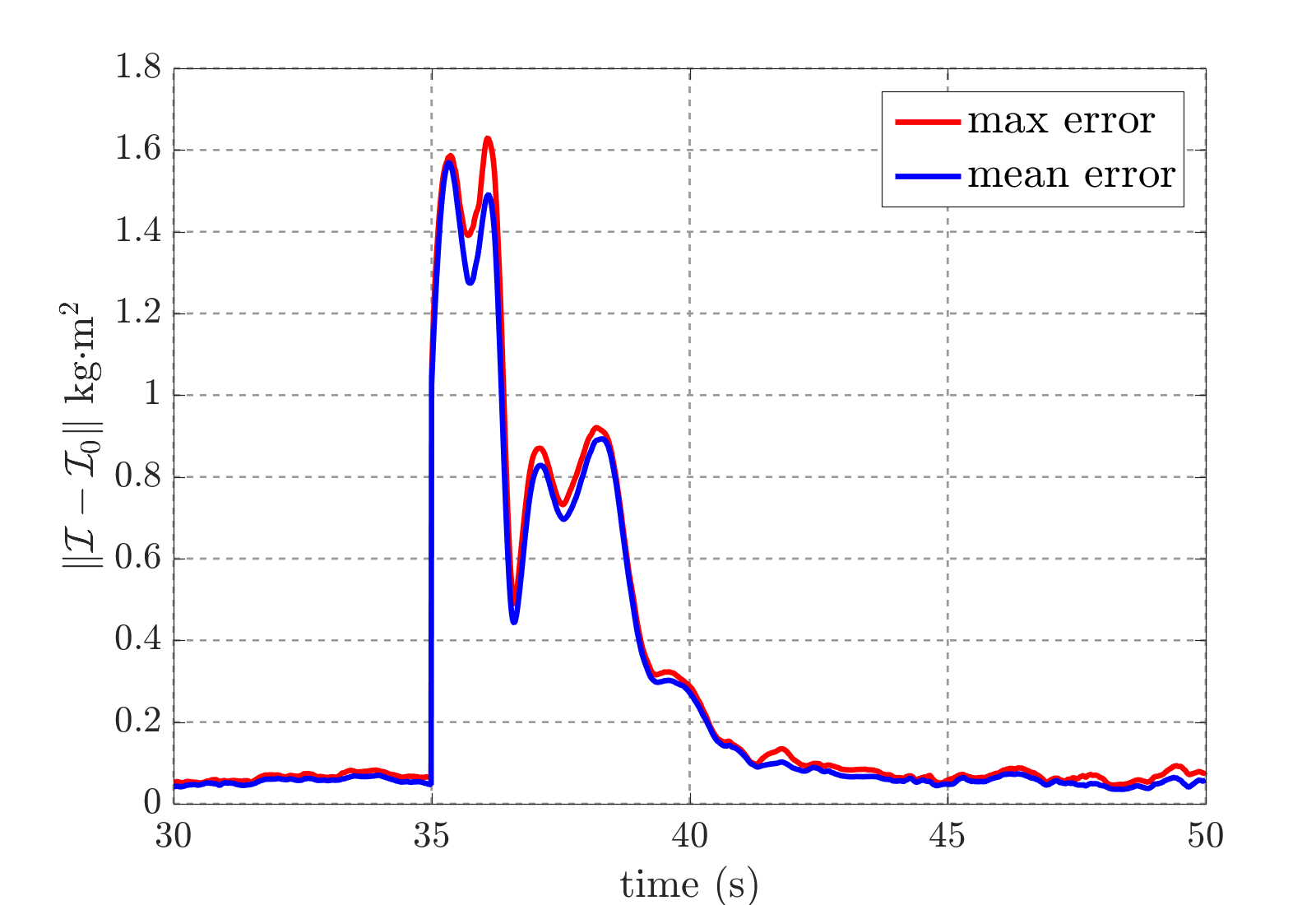}}\label{fig::J_c} &
  \subfloat[][]{\includegraphics[trim =23mm 2mm 5mm 2mm,width=0.32\textwidth]{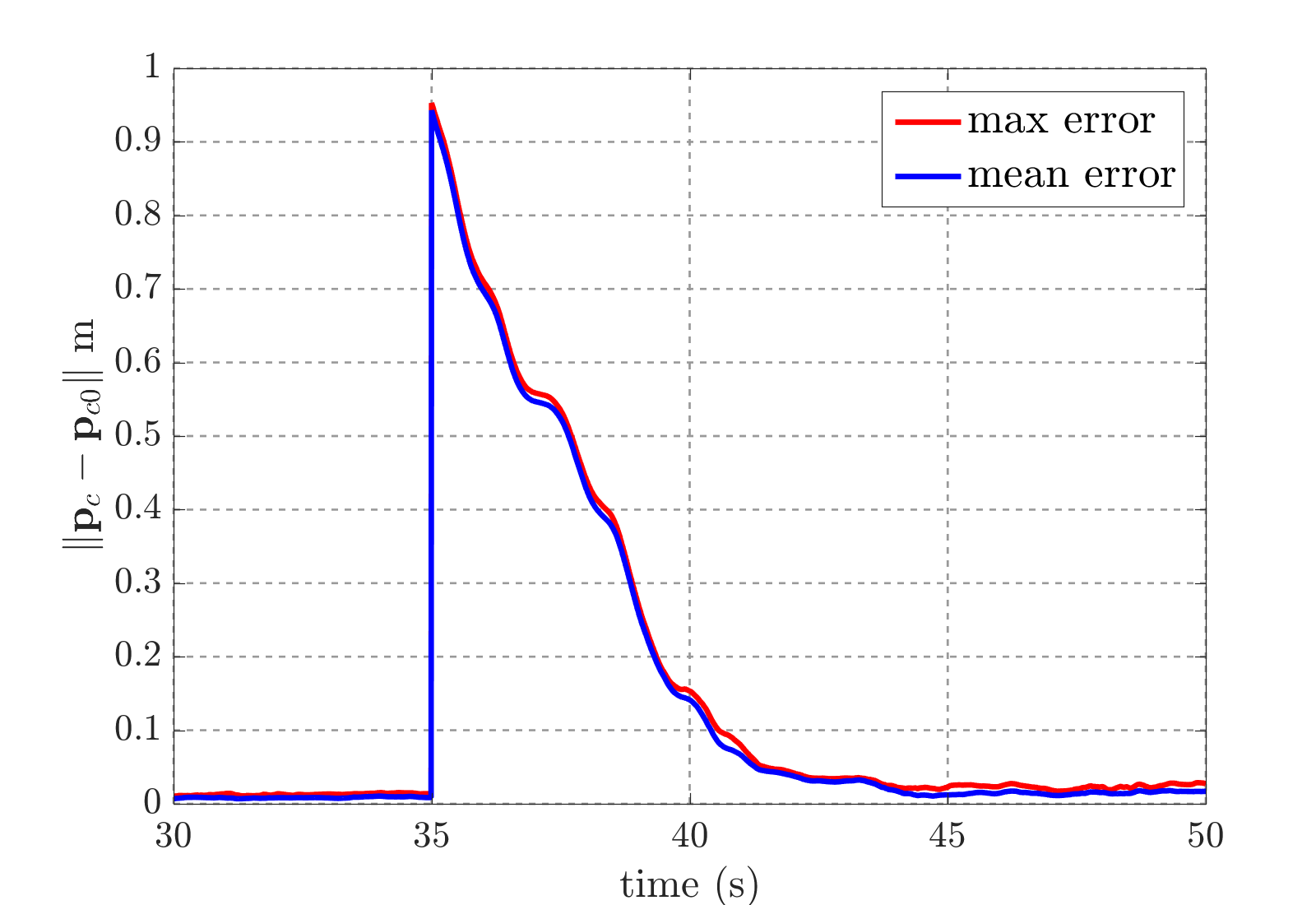}}\label{fig::pc_c} &
  \subfloat[][]{\includegraphics[trim =23mm 2mm 5mm 2mm,width=0.32\textwidth]{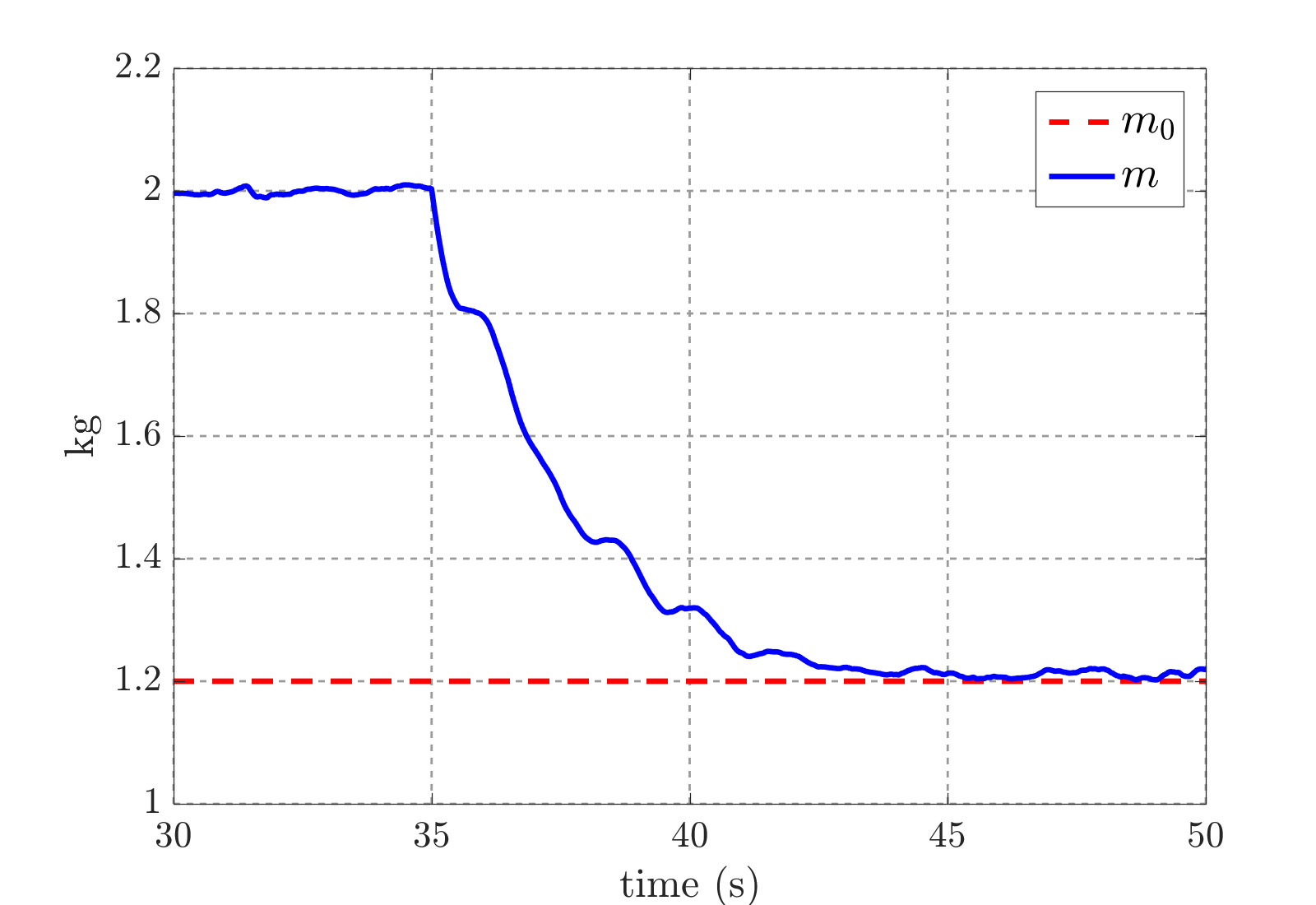}}\label{fig::m_c}
\end{tabular}
\caption{Adaptive estimation of inertia tensor $\mathcal{I}$, mass center ${\p_c}$ and mass $m$ after a sudden change of the load at $t=35$s. (a) is $\mathcal{I}$, (b) is ${\p_c}$ and (c) is $m$. At $t=50$s, the mean estimation error is $0.058\,\mathrm{kg}\cdot\mathrm{m}^2$ for inertia tensor $\mathcal{I}$, and $0.017\,\mathrm{m}$ for $\p_c$, and the estimated mass $\hat{m}=1.196\,\mathrm{kg}$ while the true $m_0=1.2\,\mathrm{kg}$.}
\label{fig::adaptive} 
\vspace{-1.5em} 
\end{figure*}
\section{Numerical Results}\label{section::num}
In numerical simulations, there are two procedures for identification: first, each robot $i$ estimates the relative pose $g_{ji}$ with its neighbours $j\in N_i$; when the estimation of $g_{ji}$ converges, each robot $i$ starts estimating $\mathcal{I}_i$, $\p_c^i$ and $m$ in its local reference frame while communicating with its neighbours to make consensus. We assume that there can be a sudden change of inertia parameters for the load during manipulation. The numerical simulation results indicate that the approach proposed is able to identify kinematic and dynamic unknown parameters with satisfactory efficiency and accuracy. According to our simulation results, the convergence speed mainly depends on measurement noise.\par
In simulation, there are $n=5$ robots manipulating a 3D rigid body and the robots are connected as a ring network. The relative pose $g$ between each robot, inertia tensor $\mathcal{I}$, mass center $\p_c$ and mass $m$ are priorly unknown. The measurement noise for $\bm{\omega}_i$, $\bm{v}_i$, $\overline{\bm{a}}_i$, $\bm{f}_i$ and $\bm{\tau}_i$ are zero-mean Gaussians with covariance $\Sigma=\delta^2 \cdot\I$ where $\delta=0.4$.
\subsection{Estimation of $g_{ji}$}\label{subsection::pose_est}
Each individual robot $i$ estimates the relative pose $g_{ji}$ with its neighbour $j\in N_i$ using linear dual quaternion observation models. In fact, for each pair $(i,\,j)\in E$, only one of $g_{ij}$ or $g_{ji}$ needs to be estimated and the other can be got by $g_{ij}g_{ji}=\I$. All the initial guesses for $g_{ji}$ are zero rotation and zero translation. The identification results are \cref{fig::err_g} and it can be seen that all robots accurately identify the relative pose with neighbours in several seconds.
\subsection{Estimation of $\mathcal{I}_i$, $\p_c^i$ and $m$}
The estimated relative pose $g_{ji}$ at $t=8$s are used for inertia parameters identification. Each robot only knows forces and torques applied by itself and can only communicate with its neighbours. The total wrench $\bm{F}_i$ and $\bm{T}_i$ is computed through consensus in different coordinates. We may also make consensus on $\mathcal{I}_i$, $\p_c^i$ and $m$, the resulting estimation errors of which may vary slightly since $g_{ji}$ are not perfectly known. As for the initial guesses $\mathcal{I}_i=\mathrm{diag}\{1,\,1,\,1\}\mathrm{kg}\cdot\mathrm{m}^2$, $\p_c^i$ is the geometric center of contact points and $m=1$kg. The initial covariance is $P_0 = 100\I$ and the forgetting factor is $\lambda=1.005$. The results are shown in \cref{fig::inertia} and it can be seen that it takes less than $15$s for the estimation to converge. \par
We also test the robustness of the approach to the sudden change of inertia, mass center and mass. In simulation, at $t=35$s after the estimation converges, there is a change of the load and the results of re-identification are \cref{fig::adaptive} which indicate that our approach may be used for adaptive cooperative manipulation.
\section{Conclusion}\label{section::conclusion}
In this paper, we present a distributed and recursive approach to online identification for rigid body cooperative manipulation. Linear observation models with local measurements are derived whose uncertainties can be explicitly evaluated under independence assumptions. We also develop dynamic consensus in different coordinates and an appropriate filter for pose estimation with dual quaternions.
\section*{Acknowldgement}
This material is based upon work supported by the National Science Foundation under Grant CNS 1329891. Any opinions, findings and conclusions or recommendations expressed in this material are those of the authors and do not necessarily reflect the views of the National Science Foundation.

%
%
\bibliographystyle{IEEEtran}
\bibliography{mybib}

\end{document}